\documentclass[12pt,reqno]{amsart}

\numberwithin{equation}{section} 

\usepackage{enumerate}
\usepackage{amssymb}
\usepackage{amsmath}
\usepackage{amscd}
\usepackage{amsthm}
\usepackage{amsfonts}
\usepackage{graphicx}
\usepackage[all]{xy}
\usepackage{verbatim}
\usepackage{hyperref}

\theoremstyle{definition}
	\newtheorem{definition}{Definition} 
	\newtheorem*{definition*}{Definition}
	
	\numberwithin{definition}{section}
\theoremstyle{plain}
	\newtheorem{lemma}[definition]{Lemma}
	\newtheorem{proposition}[definition]{Proposition}
	\newtheorem{theorem}[definition]{Theorem}
	\newtheorem*{theorem*}{Theorem}

	\newtheorem*{claim*}{Claim}
\theoremstyle{remark}

\newcommand{\cB}{\mathcal{B}}
\newcommand{\cC}{\mathcal{C}}

\newcommand{\bE}{\mathbb{E}}

\newcommand{\bP}{\mathbb{P}}
\newcommand{\bR}{\mathbb{R}}
\newcommand{\bZ}{\mathbb{Z}}

\newcommand{\bN}{\mathbb{N}}

\newcommand{\R}{\mathbb{R}}

\newcommand{\gom}{\mathfrak{m}}

\newcommand{\onto}{\xymatrix{\ar@{>>}[r]&}}
\newcommand{\da}[4]{\xymatrix{#1 \ar@<.5ex>[r]^{#2} \ar@<-.5ex>[r]_{#3} & #4}}
\newcommand{\tphi}{\widetilde{\phi}}
\newcommand{\hphi}{\widehat{\phi}}

\begin{document}
\title[Limit theorems for quasimorphisms]{Biharmonic functions on groups and limit theorems for quasimorphisms along random walks}
\author{Michael Bj\"orklund and Tobias Hartnick}
\address{Einstein Institute of Mathematics \\
Edmond J. Safra Campus \\
91904 Jerusalem \\
Israel}
\address{ETH Z\"urich \\
Departement Mathematik \\
R\"amistrasse 101 \\
8092 Zurich \\
Switzerland}
\email{michael.bjoerklund@math.ethz.ch, tobias.hartnick@math.ethz.ch}
\begin{abstract}
We show for very general classes of measures on locally compact second countable groups that every Borel measurable quasimorphism is at bounded distance from a quasi-biharmonic one. This allows us to deduce non-degenerate central limit theorems and laws of the iterated logarithm for such quasimorphisms along regular random walks on topological groups using classical martingale limit theorems of Billingsley and Stout. For quasi-biharmonic quasimorphism on countable groups we also obtain integral representations using martingale convergence.
\end{abstract}
\maketitle

\section{Introduction}
\subsection{Preliminaries}
Throughout this article we denote by $G$ a topological group and by $\mu$ a regular probability measure on $G$. We then refer to the pair $(G,\mu)$ as a \emph{measured group}. A real-valued function $\phi$ on $G$ is called a \emph{quasimorphism} if the differential 
\[
\partial \phi(g,h) := \phi(gh) - \phi(g) - \phi(h),
\]
is bounded on $G \times G$. In this case we denote by $D(\phi) := \sup_{g,h} | \partial \phi(g,h) |$ 
the \emph{defect} of $\phi$.

In this article we will be interested in the behavior of quasimorphisms on measured groups along certain random walks. More precisely, let $(\Omega,\cB,\bP)$ be a probability measure space, and $(\omega)_n$ a sequence of $\mu$-distributed independent
$G$-valued measurable maps on $\Omega$. We then define for all $n \geq 1$ a measurable map $z_n$ by $z_n := \omega_0 \cdots \omega_{n-1}$ and refer
to the sequence $(z_n)$ as a \emph{random walk on $G$ with $\mu$-distributed increments}.  We will then be concerned with the asymptotic distributional and almost sure behavior of the sequence $(\phi(z_n))$ for a $\mu$-integrable quasimorphism $\phi$ on $G$.

In order to state our main result, we introduce the following terminology and notations: We denote 
by $\mu^{*n}$ the $n$'th convolution of $\mu$, and by $G_\mu$ the closed semigroup generated by the
support of $\mu$. We say that $\mu$ is \emph{symmetric} if $\iota_*\mu = \mu$, where $\iota$ is the involution of $G$ given by $\iota(g) = g^{-1}$. Moreover, if $G$ is locally compact then $\mu$ is called \emph{spread out} if some convolution power $\mu^{*n}$ is non-singular with respect to the Haar measure on $G$.

Given a $\mu$-integrable quasimorphism $\phi$ we define its \emph{$\mu$-distortion} $\ell_\mu(\phi)$ by the formula 
\[
\ell_\mu(\phi) = \lim_n \frac{1}{n} \int_G \phi \, d\mu^{*n}.
\]
We say that $\phi$ is \emph{$\mu$-tame} if there is a constant $C$ such that 
\[
| \phi(g) - n \ell_\mu(\phi) | \leq C,
\]
for $\mu^{*n}$-almost every $g$ and all $n$. In particular, if $\mu$ is symmetric, then $\ell_\mu(\phi) = 0$  so that in this case $\phi$ is $\mu$-tame if
and only if it is $\mu^{*n}$-essentially bounded on $G_\mu$ for all $n$.

\subsection{Statement of the main results}
Let $(G,\mu)$ be a measured group, and suppose that $(z_n)$ is a random walk on $G$ with $\mu$-distributed increments. A $\mu$-measurable 
quasimorphism $\phi$ on $G$ \emph{satisfies the central limit theorem} with respect to $\mu$ if there exists a non-negative constant $\sigma$, such that the sequence
\[
\Big( \, \frac{\phi(z_n) - n \ell_\mu(\phi)}{\sqrt{n}} \, \Big)
\]
converges in the sense of distributions to a centered Gaussian probability distribution with variance $\sigma^2$. We say that the central limit theorem is 
\emph{non-degenerate} if $\sigma > 0$. Moreover, if the limit superior
\[
\varlimsup_{n \to \infty} \frac{\phi(z_n) - n \ell_\mu(\phi)}{\sqrt{n \log \log{n}}} = \sigma > 0,
\]
holds almost everywhere, we say that $\phi$ \emph{satisfies the law of the iterated logarithm} with respect to $\mu$. We can now state the first version of our main theorem:

\begin{theorem} \label{mainZFC}
Let $G$ be a locally compact second countable topological group, $\mu$ a spread-out  probability measure on $G$ and $\phi$ a Borel measurable quasimorphism 
on $G$, which is square-integrable with respect to $\mu$. Then $\phi$ satisfies the central limit theorem with respect to $\mu$. The central limit theorem is non-degenerate if and only if $\phi$ is not $\mu$-tame. In this case, $\phi$ also satisfies the law of the iterated logarithm with 
respect to $\mu$.
\end{theorem}
We remark that the assumptions on $\mu$ and $\phi$ are very weak. For example, if $\mu$ has exponential moments, then every measurable quasimorphism will be $\mu$-square-integrable so that the theorem applies.\\

It is possible to extend Theorem \ref{mainZFC} even further to arbitrary measured groups if one is willing to make additional assumptions concerning the model of set theory that one uses. For example, assuming the continuum hypothesis (CH), we can prove:
\begin{theorem}[CH]\footnote{For the convenience of the reader not interested in this generalization we have marked all results in this paper which depend on the continuum hypothesis by (CH).}\label{main}
Let $(G,\mu)$ be a measured group and suppose that $\phi$ is a Borel measurable quasimorphism 
on $G$, which is square-integrable with respect to $\mu$. If $\phi$ is not $\mu$-tame, then 
it satisfies the non-degenerate central limit theorem and the law of the iterated logarithm with 
respect to the measure $\mu$.
\end{theorem}
In fact, it would suffice to assume Martin's axiom in descriptive set theory, which is weaker then (CH). However, we do not know how to prove Theorem \ref{main} within ZFC.

\subsection{General strategy}
Let us call two quasimorphisms $\phi$ and $\tphi$ \emph{equivalent} if their difference is uniformly bounded. The results we are interested in are invariant under replacing a quasimorphism by an equivalent one; in particular, $\ell_\mu(\phi)$ does not change upon changing $\phi$ to an equivalent representative. Thus for the proofs of Theorem \ref{mainZFC} and Theorem \ref{main} it is convenient to work with a representative of the equivalence class of $\phi$, which is adapted to the measure $\mu$. The following definition makes this idea precise:
\begin{definition}
Let $(G,\mu)$ be a measured group. Let $(G,\mu)$ be a measured group. A real-valued function $\tphi$ on $G$ is \emph{quasi-right-harmonic with
respect to $\mu$} if it is Borel measurable and integrable with respect to the measure $\mu$ and there is a constant $\ell$ such that
\[
\int_G \tphi(gh) \, d\mu(h) = \tphi(g) + \ell,
\]
for all $g$. If $\ell = 0$, we say that $\tphi$ is \emph{right-harmonic}. There is also an obvious 'left'-notion of quasi-harmonicity which we
will adapt, and we say that $\tphi$ is \emph{quasi-biharmonic with respect to $\mu$} if it is both quasi-left- and quasi-right-harmonic with 
respect to $\mu$. 
\end{definition}
Under the assumptions of Theorem \ref{mainZFC} (or Theorem \ref{main} and (CH)) we are able to establish the existence of a $\mu$-quasi-biinvariant quasimorphism $\tphi$ equivalent to $\phi$. In the situation of Theorem \ref{mainZFC} it is even unique up to a constant, and assuming (CH) it can be chosen universally measurable.\\

From the existence of $\tphi$ one deduces immediately that for a random walk $(z_n)$ on $G$ with $\mu$-distributed increments the sequence $(\phi(z_n) - n \ell_\mu(\phi))$ can be uniformly approximated by a cocycle martingale. Our square-integrality assumption on $\phi$ implies that this cocycle martingale can be chosen square-integrable as well. Then our results follow from the fundamental works of Billingsley 
and Stout on square-integrable cocycle martingales. 

\subsection{Integral Representations} Assume now that the measured group $(G, \mu)$ in question is countable (and the topology on $G$ is the discrete one). Then, by the results established in the proof of the two main theorems, every quasimorphism $\phi$ on $G$ is at bounded distance to a unique $\mu$-quasi-biharmonic quasimorphism $\tphi$ subject to the normalization $\tphi(e) = 0$, and this can be shown without invoking (CH). If we assume that $\phi$ is cohomologically non-trivial, i.e. not at bounded distance from a homomorphisms, then we can establish an integral representation theorem for $\tphi$:
\begin{theorem}\label{IntRepIntro} Let $(G, \mu)$ be a countable measured group, $\phi$ a cohomologically non-trivial quasimorphism on $G$ and $\tphi$ as above. Then there exists a non-atomic $G$-quasi-invariant Lebesgue space $(X, \nu)$ and an additive cocycle $\alpha: G \times X \to \R$ such that
\[\tphi(g) = \int_{X} \alpha(g,\xi) \, d\nu(\xi) \quad(g \in G).\]
\end{theorem}
A more precise version of this theorem will be proved in Proposition \ref{IntRepMain} below.

\subsection{Connection to existing work}
The ideas developed in the present paper allow for a short proof of an (unpublished) conjecture of Calegari on central limit theorems 
for quasimorphisms along random walks on countable groups. Partial results in this direction, complemented with a wide range of 
applications, can be found in a recent preprint \cite{CaMa10} by Calegari and Maher, which is heavily inspired by an earlier paper 
\cite{CaFu10} by Calegari and Fujiwara. We also mention the works of Horsham \cite{Hor08} and Horsham-Sharp \cite{HoSh09} 
where more precise limit theorems are established for quasimorphisms along (not necessarily independent) random walks on free groups
and certain surface groups.

The results on integral representations of quasimorphisms are inspired from the work of Picaud in 
\cite{Pic97}. In trying to extend his results to more general groups, we came up with the idea to study quasimorphisms along random trajectories. These attempts ultimately triggered most of the results in this paper.

\subsection{Outline of the paper}
The paper is organized as follows. We start in Section \ref{Quasiharmonic} by providing the existence of a suitable quasi-biharmonic quasimorphisms equivalent to a given one under the assumptions of either Theorem \ref{mainZFC} of Theorem \ref{main}. In the former case, the existence is established using a standard Banach limit argument going back to Burger and Monod \cite{BuMo}. In the latter case we make use of measure-linear means. A beautiful construction of such means assuming (CH) was provided by Mokobodzki; since there are few accessible accounts of Mokobodzki's 
construction, we explain the necessary details in an appendix, expanding the classical presentation in \cite{Mey73}. In Section \ref{cocycle martingales} we provide the proofs of Theorem \ref{mainZFC} and Theorem \ref{main} by reduction to the case of cocycle 
martingales, for which a quite satisfactory central limit theory is known. Finally, Section \ref{cocycles} is devoted to establish the integral representation result stated in Theorem \ref{IntRepIntro} and some variants thereof.

\subsection*{Acknowledgements}
We thank ETH Z\"urich, Hebrew University and Technion for hospitality  during the preparations of this article, and 
Uri Bader, Marc Burger, Danny Calegari, Albert Fisher, Matt Foreman, Nicolas Monod, Amos Nevo, Jean-Claude Picaud, Jean-Paul Thouvenot and 
Benji Weiss for many stimulating discussions on the subject of this article. T. H. was partially supported by SNF
grant PP002-102765.

\section{Quasi-biharmonic representations of quasimorphisms} \label{Quasiharmonic}

We now aim to construct quasi-biharmonic quasimorphisms equivalent to a given quasimorphism. Our first observation concerns the possible values of the additive constant implicit in the term quasi-biharmonic:
\begin{lemma} \label{quasi-harmonic}
Let $(G,\mu)$ be a measured group, and suppose that $\phi$ is quasi-right-harmonic quasimorphism with 
respect to $\mu$. Then, 
\[
\int_G \phi(gh) \, d\mu(h) = \phi(g) + \ell_\mu(\phi),
\]
for all $g \in G$, where $\ell_\mu$ is the $\mu$-distortion of $\phi$.
\end{lemma}
\begin{proof}
We can without loss of generality assume that $\phi(e) = 0$. Hence, for all $n$, we have
\[
\int_G \phi(h) \, d\mu^{*n}(h) = n \ell,
\]
which immediately gives the result. 
\end{proof}
We can now prove the desired existence result in the situation of Theorem \ref{main}. For this we use the existence of measure-linear means (due to Mokobodzki), as explained in detail in the appendix of this paper. As far as the statement of the present proposition is concerned we recall that given a measurable space $(Y,\cB)$ a subset $A \subseteq Y$ is called \emph{universally measurable} if it is measurable
with respect to every complete finite measure on the $\cB$. If $(Z,\cC)$ is another 
measurable space, we say that a map $f : Y \to Z$ is universally measurable if the pre-image of every 
universally measurable set in $Z$ is universally measurable in $Y$. For a more detailed discussion of these terms we refer the reader to the appendix.
\begin{proposition}[CH]\label{quasi-biharmonic}
Let $(G,\mu)$ be a measured group, and $\phi$ a Borel measurable quasimorphism, integrable with respect 
to $\mu$. Then there is a universally measurable quasimorphism $\tphi$, which is equivalent 
to $\phi$, and quasi-biharmonic with respect to $\mu$, i.e.
\[
\int_G \tphi(gk) \, d\mu(k) = \int_G \tphi(kg) \, d\mu(k) = \tphi(g) + \ell_\mu(\phi),
\]
where $\ell_\mu(\phi)$ is the $\mu$-invariant of $\phi$. 
\end{proposition}

\begin{proof}
Recall that the homogeneous representative of $\phi$ is given by
\[
\hphi(g) = \lim_{n \to \infty} \frac{\phi(g^n)}{n}.
\]
Since $\phi$ is assumed to be Borel measurable, and multiplication is a continuous operation, we see that 
the functions $g \mapsto \phi(g^n)$ are Borel measurable for all $n$. Since $\hphi$ is the pointwise limit of
a sequence of Borel measurable functions, it is itself Borel measurable. By assumption, the measure $\mu$
is a \emph{regular} Borel probability measure on $G$, and thus the functions 
\[
\tphi_n(g_1,g_2) = \int_G \partial \hphi(g_1,g_2 h) \, d\mu^{*n}(h),
\]
are all Borel measurable, and uniformly bounded in absolute value by the defect $D(\hphi)$. Moreover, \emph{every}
section $g \mapsto \tphi_n(g,k)$, for fixed $k \in G$, is Borel measurable and uniformly bounded on $G$ for all $n$. 
Thus, the conditions in Proposition \ref{mokobodski} are satisfied, and the function
\[
\tphi(g) := \hphi(g) + \int_{\bN} \tphi_n(g,e) \, d\gom(n),
\]
is well-defined and universally measurable on $G$. Moreover, it is equivalent to $\hphi$ and hence to $\phi$. 
We claim that $\tphi$ is the quasi-biharmonic function we are looking for. To verify this, we first observe that Proposition 
\ref{mokobodski} applied to the probability measures $\delta_g \times \mu$, for $g \in G$, yields the identity
\[
\int_{G \times G} \Big( \int_{\bN} \tphi_n(g_1,g_2) \, d\gom(n) \Big) \, d(\delta_g \times \mu)(g_1,g_2) = 
\int_{\bN} \Big( \int_{G} \tphi_n(g,h) \, d\mu(h) \Big) \, d\gom(n).
\]
Since,
\begin{eqnarray*}
\int_G \tphi_n(g,h) \, d\mu(h) 
&=& 
\int_G \int_G \big( \hphi(ghk) - \hphi(hk) - \hphi(g) \big) \, d\mu^{*n}(k) \, d\mu(h) \\
&=&
\int_G \big( \hphi(gh) - \hphi(h) - \hphi(g) \big) \, d\mu^{*(n+1)}(h) \\
&=&
\tphi_{n+1}(g,e),
\end{eqnarray*}
for all $n$, we note that by the shift-invariance of $\gom$, 
\[
\int_G \Big( \int_{\bN} \tphi_n(g,h) \, d\gom(n) \, \Big) \, d\mu(h) = \int_{\bN} \tphi_n(g,e) \, d\gom(n).
\]
Moreover, the identity
\[
\tphi_n(gk,e) = \tphi_n(g,k) + \tphi_n(k,e) - \hphi(k) - \hphi(gk)
\]
holds for all $g, k \in G$ and $n$, which readily implies that
\[
\int_G \tphi(gk) \, d\mu(k) = \tphi(g) + \int_G \tphi(k) \, d\mu(k), \quad \forall \, g \in G.
\]
By Lemma \ref{quasi-harmonic}, we can now conclude that 
\[
\int_G \tphi(k) \, d\mu(k) = \ell_\mu(\phi). 
\]
Since $\hphi(gh) = \hphi(hg)$ for all $h,g \in G$, the proof of the quasi-left-harmonicity of $\tphi$ boils down to 
establishing the identity, 
\[
\int_G \Big( \int_{\bN} \phi_n(k,g) \, d\gom(n) \, \Big) \, d\mu(k) = \ell_\mu(\phi) - \int_G \hphi(k) \, d\mu(k),
\]
for all $g \in G$. To do this, we first notice that 
\[
\lim_{n \to \infty} \frac{1}{n} \int_G \hphi(gh) \, d\mu^{*n}(h) = \ell_\mu(\phi)
\]
for all $g \in G$. Hence, for fixed $g \in G$, 
\begin{eqnarray*}
\int_G \Big( \int_{\bN} \phi_n(k,g) \, d\gom(n) \, \Big) \, d\mu(k) 
&=& 
\int_{\bN} \Big( \ell_\mu^{(n+1)} (g,\hphi)- \ell_\mu^{n} (g,\hphi) - \int_G \hphi(k) \, d\mu(k)  \Big) \, d\gom(n) \\
&=&
\lim_{n \to \infty}  \frac{1}{n} \int_G \hphi(gh) \, d\mu^{*n}(h) - \int_G \hphi(k) \, d\mu(k),
\end{eqnarray*} 
by the definition of $\gom$, where
\[
\ell_\mu^{n} (g,\hphi) :=  \int_G \hphi(gh) \, d\mu^{*n}(h), \quad g \in G. 
\] 
Indeed, the expression
\[
\int_{\bN} \big( \ell_\mu^{(n+1)} (g,\hphi)- \ell_\mu^{n} (g,\hphi) \big) \, d\gom(n) 
\]
is sandwiched between the limit inferior and limit superior of the sequence
\[
\frac{1}{n} \sum_{k=1}^{n} \big( \ell_\mu^{(k+1)} (g,\hphi)- \ell_\mu^{k} (g,\hphi) \big),
\]
which is in fact convergent, with the limit 
\[
\ell_\mu(\phi) = \lim_{n \to \infty} \, \frac{1}{n} \int_G \hphi(gh) \, d\mu^{*n}(h), \quad \forall \, g \in G,
\]
which proves the existence of $\tphi$. 
\end{proof}
In order to establish a corresponding result in the situation of Theorem \ref{main} we follow an argument going back to \cite{BuMo} in the case of a countable discrete group with symmetric measure. (The fact that this argument applies in our setting was kindly pointed out to us by Nicolas Monod.) We remind the reader that if $\mu$ is a spread-out regular probability measure on a locally-compact second countable topological group $G$ and the equation $\psi * \mu = \psi$ admits an almost everywhere defined 
Haar measurable solution, which is assumed to be bounded on compact subsets of the group, then an everywhere defined 
continuous solution $\psi'$ can be found in the same Haar class as $\psi$ (see e.g. \cite{Aze70}). The same arguments carry over to the quasi-harmonic case. 
\begin{proposition} \label{ZFC-alternative}
Suppose that $\mu$ is a spread-out probability measure on $G$, and $\phi$ is a Borel measurable quasimorphism, which 
is $\mu$-integrable with respect to $\mu$. Then there is a \emph{unique} continuous quasimorphism $\tphi$, which is 
quasi-biharmonic with respect to $\mu$, satisfies $\tphi(e) = 0$ and is equivalent to $\phi$.
\end{proposition}
\begin{proof} As far as existence of $\tphi$ is concerned, large parts of the proof are analogous to the proof of Proposition \ref{quasi-biharmonic}, thus we will only highlight those steps which are used to circumvent the use of Mokobodzki means. In complete analogy with the proof of Proposition \ref{quasi-biharmonic}, we define the sequence
\[
\psi_n(g) := \int_G \partial \hphi(g,h) \, d\mu^{*n}(h), 
\]
for $n \geq 1$. The functions $(\psi_k)$ are uniformly bounded in the $L^{\infty}(G,m)$-norm, so 
by Koml\'os' theorem \cite{Kom67} there is a subsequence $(n_k)$ such that the limit
\[
\tphi(g) = \hphi(g) + \lim_{N \to \infty} \frac{1}{N} \sum_{k=1}^n \frac{1}{n_k} \sum_{j=0}^{n_k-1} \psi_j(g)
\] 
exists on a $m$-conull subset $Y \subseteq G$, so $\tphi - \hphi$ is a well-defined element in $L^\infty(G,m)$. Moreover,
by the same arguments as in the proof of Proposition \ref{quasi-biharmonic}, we see that $\tphi$ satisfies
the convolution equations, 
\[
\tphi * \mu = \mu *\tphi = \tphi + \ell_\mu(\phi)
\]
$m$-almost everywhere. Thus, by the remark preceding the proof, as $\tphi$ is clearly bounded on compact subsets of $G$, it admits a 
continuous representative (with respect to $m$) which is also a quasimorphism.

For the uniqueness part, we owe the following beautiful argument to Uri Bader: Suppose that $\tphi$, $\tphi'$ are two representatives and consider $f := \tphi -\tphi'$. Then $f$ is $\mu$-biharmonic and uniformly bounded, i.e. $f \in L^\infty(\Gamma)^{\mu \times \mu}$. However, if we denote by $(\partial_\mu G, \nu)$ the Poisson boundary of the measured group $(G, \mu)$ (see e.g. \cite{KaimVer}) and by $H^\infty_{\mu \times \mu}(G \times G)$ the space of harmonic functions on the measured group $(G \times G, \mu \times \mu)$ then we have
\begin{eqnarray*}
L^\infty(G)^{\mu \times \mu} &=& (L^\infty(G\times G)^G)^{\mu \times \mu}
= (L^\infty(G\times G)^{\mu \times \mu})^G\\
&=&H^\infty_{\mu \times \mu}(G \times G)^G = L^\infty(\partial_\mu G \times \partial_\mu G)^G\\
&=& \R,
\end{eqnarray*}
where in the last step we used double ergodicity of the $G$-action on the Poisson boundary (see e.g. \cite{Kai03}). We deduce that $f$ is constant; then $f(0) = 0$ yields $\tphi =\tphi'$.
\end{proof}

\section{Cocycle martingales and the proofs of the main theorems} \label{cocycle martingales}
Throughout this section we use the following notations: Let $(G,\mu)$ be a measured group, and $(\omega_n)$ a bi-infinite sequence of independent and $\mu$-distributed $G$-valued
measurable maps, defined on a common probability measure space $(\Omega,\cB,\bP_\mu)$, which we assume is equipped with 
an ergodic $\bZ$-action $\tau$ with the property that $\omega_n = \omega_0 \circ \tau^n$ for all $n$. We stress that any sequence
of independent and $\mu$-distributed random variables admits such a representation. For $n \geq 1$, we define the corresponding random walk by $z_n = \omega_0 \cdots \omega_{n-1}$, and for $m \leq n$ we denote by 
$\cB_m^n$ the $\sigma$-algebra generated by the maps $\{ \omega_m \ldots \omega_{n-1} \}$. 
\begin{definition}
A sequence of real-valued measurable
maps $(\rho_n)$ on $\Omega$ such that $\rho_n$ is $\cB^n_m$-measurable for a fixed $m$ and all $n \geq m$, and
\[
\bE[ \rho_{n+1} \, | \, \cB^{n}_m ] = \rho_n, \quad \forall \, n \geq m
\] 
is called a \emph{martingale with respect to $(\cB^n_m)$}. 
We say that $(\rho_n)$ is \emph{centered} if $\int_\Omega \rho_m \, d\bP_\mu = 0$. We say that a sequence $(\eta_n)$ of real-valued measurable maps  on $\Omega$ is a \emph{cocycle} for $\tau$ if
\[
\eta_{n+m} = \eta_n + \tau^m \eta_n, \quad \forall \, m,n \geq 0,
\]
and if $(\eta_n)$ in addition is a martingale with respect to some filtration $(\cB_m^n)$ of $\cB$, we say that $(\eta_n)$ is
a \emph{cocycle martingale}. If $\eta_0$ is not constant almost everywhere, the cocycle martingale is \emph{non-degenerate}. 
\end{definition}

The asymptotic almost sure and distributional behaviors of cocycle martingales have been extensively studied in the literature. In order to be able to apply these results in our context, we establish the following approximation result:

\begin{proposition}[CH] \label{cocycle martingale}
Let $(G,\mu)$ be a measured group, and suppose that $(z_n)$ is a random walk on $G$ with $\mu$-distributed 
increments. For any Borel measurable square-integrable quasimorphism $\phi$, there is a square-integrable and centered cocycle martingale 
$(\rho_n)$ with respect to the $\sigma$-algebra filtration $(\cB_{-\infty}^{n})$ such that
\[
\sup_{n} \| \phi(z_n) - n \ell_\mu(\phi) - \rho_n \|_{\infty} \leq 3D(\phi),
\]
Moreover, the cocycle martingale $(\rho_n)$ is non-degenerate if and only if $\phi$ is not $\mu$-tame. 
\end{proposition}
Alternatively, we have the following version of the proposition which holds within ZFC:
\begin{proposition} \label{cocycle martingaleZFC}
Let $G$ be a locally compact second countable topological group and $\mu$ a spread-out regular probability measure on $G$. Suppose that $(z_n)$ is a random walk on $G$ with $\mu$-distributed 
increments. For any Borel measurable square-integrable quasimorphism $\phi$, there is a square-integrable and centered cocycle martingale 
$(\rho_n)$ with respect to the $\sigma$-algebra filtration $(\cB_{-\infty}^{n})$ such that
\[
\sup_{n} \| \phi(z_n) - n \ell_\mu(\phi) - \rho_n \|_{\infty} \leq 3D(\phi),
\]
Moreover, the cocycle martingale $(\rho_n)$ is non-degenerate if and only if $\phi$ is not $\mu$-tame. 
\end{proposition}

\begin{proof}[Proof of Proposition \ref{cocycle martingale} and Proposition \ref{cocycle martingaleZFC}]
Appealing either to Proposition \ref{quasi-biharmonic} or to Proposition \ref{ZFC-alternative} we find a quasi-biharmonic quasimorphism $\tphi$ equivalent to $\phi$. Now for every $g \in G$ the function $\psi_g(h) = \tphi(hg) - \tphi(h)$ is bounded and left-harmonic with respect to $\mu$. For any $\omega = (\omega_n) \in \Omega$, we define the sequence 
\[
q_k = \omega_{-k+1} \cdots \omega_{-1}, \quad k \geq 0.
\] 
By the left-harmonicity of $\psi_g$, the sequence $\xi_k^n := \psi_{z_n}(q_k)$ is a bounded martingale with respect to the $\sigma$-algebra filtration $(\cB_{-k}^{n})$, for every fixed $n$. 
Thus, by the martingale convergence theorem, the limit 
\[
\eta_n = \lim_{k \to \infty} \big( \, \tphi(q_k \, z_n) - \tphi(q_k) \, \big) 
\]
exists on a conull subset $\Omega' \subseteq \Omega$, which we can take to be independent of $n$. We claim that the sequence $(\eta_n)$ is a cocycle with respect to the shift-action on $\Omega$, i.e. 
\[
\eta_{m+n}(\omega) = \eta_m(\omega) + \eta_n(\tau^m \omega),
\]
for all $m, n \geq 0$, where $\tau$ denotes the shift action on $\Omega$. To check this, we notice that
\begin{eqnarray*}
\eta_n(\tau^m\omega) 
&=& 
\lim_{k \to \infty} \big( \, \tphi(\omega_{-(k-m)} \cdots \omega_{-1} \omega_o \cdots \omega_{m-1} \omega_m \cdots \omega_{n+m-1}) - \\
&-& 
\tphi(\omega_{-(k-m)} \cdots \omega_{-1} \omega_o \cdots \omega_{m-1}) \, \big) \\
&=&
\lim_{k \to \infty} \big( \, \tphi(\omega_{-k} \cdots \omega_{-1} \omega_o \cdots \omega_{m-1} \omega_m \cdots \omega_{n+m-1}) - \\
&-& 
\tphi(\omega_{-k} \cdots \omega_{-1} \omega_o \cdots \omega_{m-1}) \, \big) \\
&=& 
\eta_{n+m}(\omega) - \eta_{m}(\omega),
\end{eqnarray*}
for all $m, n \geq 0$, and hence $\eta_n = \sum_{k=0}^{n-1} \tau^k \eta_1$ for all $n$. In particular, 
\[
d_n = \eta_{n+1} - \eta_{n} = \tau^n \eta_1,
\]
so 
\[
\bE[ \, d_n \, | \, \cB_{-\infty}^n \, ](\omega) = \int_G \tphi(\omega_{-k} \cdots \omega_{n-1} g) \, d\mu(g) - \tphi(\omega_{-k} \cdots \omega_{n-1}) = \ell_\mu(\phi),
\]
almost everywhere, by the quasi-right-harmonicity of $\tphi$. We conclude that $\rho_n = \eta_n - n\ell_\mu(\phi)$ is a centered cocycle martingale, which is 
clearly square-integrable with respect to the induced Bernoulli measure $\bP_\mu$ on $\Omega$.  

To prove the claim about the non-degeneracy of $(\rho_n)$, we note that $\rho_1$ is almost everywhere zero on $\Omega$ if and only if 
\[
\bE[ \, \eta_1 \, | \, \cB_{-m}^{0} \,](\omega) = \tphi(q_m \omega_0) - \tphi(q_m) = \ell_\mu(\phi),
\]
almost everywhere on $\Omega$ and for all $m$. Since $\phi$ is on a uniformly bounded distance from $\tphi$, this condition implies that 
\[
| \phi(g) - n \ell_\mu(\phi) | \leq 3D(\phi),
\]
for $\mu^{*m}$-almost every $g$ and all $m$, which means that $\phi$ is $\mu$-tame. Conversely, if $\phi$ is $\mu$-tame, 
then clearly $\rho$ is zero almost everywhere, which proves the claim. 
\end{proof}
Using either Proposition \ref{cocycle martingaleZFC} or Proposition \ref{cocycle martingale}
we can now reduce the proofs of Theorem \ref{mainZFC} and Theorem \ref{main} to two fundamental results in the theory of cocycle martingales.  
\begin{proof}[Proof of Theorem \ref{mainZFC} and Theorem \ref{main}]
If $\phi$ is $\mu$-tame, then  the distribution of $( \phi(z_n) - n \ell_\mu(\phi) )$ is contained in a fixed bounded subset of $\R$ for all $n$, and thus the central limit theorem holds with $\sigma=0$. Now assume that $\phi$ is not $\mu$-tame. Then, by Proposition \ref{cocycle martingaleZFC} or 
Proposition \ref{cocycle martingale},  there is a centered non-degenerate cocycle martingale $(\rho^0_n)$ on a uniformly bounded distance away from $\phi(z_n) - n \ell_\mu(\phi)$.\\

A theorem of Billingsley \cite{Bil61} asserts that a centered  cocycle martingale $(\rho_n)$ on $(\Omega,\cB,\bP_\mu)$ with 
$\sigma = \| \rho_0 \|_2 < \infty$  satisfies a central limit theorem, i.e. the sequence $(\rho_n / \sqrt{n})$ converges 
in the sense of distributions to a centered Gaussian distribution with variance $\sigma$, and if $(\rho_n)$ is non-degenerate then 
$\sigma > 0$. Assuming the same conditions, Stout \cite{Sto70} later proved that the iterated law of the logarithm holds too, i.e. 
there is a conull subset $\Omega' \subseteq \Omega$ such that
\[
\varlimsup_n \frac{\rho_n(\omega)}{\sqrt{n \log \log{n}}} = \sigma,
\]
for all $\omega \in \Omega'$. Applying these two results to the cocycle martingale $(\rho^0_n)$ above, we obtain the desired conclusion for $\phi$.
\end{proof}

\section{Quasimorphisms on countable groups and integral representations} \label{cocycles}

In this section we focus on the special case of quasimorphisms on discrete measured groups; in order to avoid the use of the continuum hypothesis we assume that the group in question is countable. Thus let $G$ be a countable group, $\mu$ a regular probability measure (which is automatically spread-out) and $\phi$ a $\mu$-integrable quasimorphim on $G$. For technical reasons we will always assume that the support $\mu$ generates $G$ as a semigroup. By Proposition \ref{ZFC-alternative} there then exists a unique $\mu$-quasi-biharmonic quasimorphism $\tphi$ equivalent to $\phi$ with $\tphi(e) = 0$. We want to write $\tphi$ as an integral over some special cocycle on a non-atomic Lebesgue space. The precise type of cocycle we have in mind is described in the following definition:

\begin{definition}
Let $(X,\nu)$ be a $G$-space, i.e. a probability measure space equipped with a measurable action of $G$. We say that $(X,\nu)$ is
\emph{quasi-invariant} if the action of $G$ preserves the measure class of $\nu$, i.e. if $g_*\nu$ is equivalent to $\nu$ for all $g$. A $\nu$-measurable map $\alpha : G \times X \to \bR$ is called a \emph{cocycle} on $(X,\nu)$ if 
\[
\alpha(gh,\xi) = \alpha(g,h \xi) + \alpha(h,\xi),
\]
for all $g,h \in G$ and $\xi \in X$. Given a quasimorphism $\phi$ on $G$ and a cocycle $\alpha : G \times X \to \bR$, we call $\alpha$ a \emph{cocycle representative of $\phi$} if there is a constant such that 
\[
\| \phi(g) - \alpha(g,\cdot) \|_{L^{\infty}(X,\nu)} \leq C,
\]
for all $g \in G$.
\end{definition}
Note that if $\alpha : G \times X \to \bR$ is a cocycle representative of some quasimorphism $\phi$, then $\phi_\xi(g) := \alpha(g,\xi)$ is a quasimorphism on $G$, equivalent to $\phi$, for $\nu$-almost all $\xi \in X$, hence the name.\\

Now that we have clarified our notion of cocycle let us turn to the measure spaces that we are going to use: Furstenberg \cite{Fur63} constructed, for any countable measured group $(G,\mu)$, a quasi-invariant $G$-space 
$(\partial_\mu G,\nu)$, called the \emph{Poisson boundary} of $(G,\mu)$, where $\nu$ satisfies that convolution equation,
\[
\int_G \int_{\partial_\mu G} \psi(g\xi) \,  d\nu(\xi) d\mu(g) = \int_{\partial_\mu G} \psi(\xi) \, d\nu(\xi),
\]
for all $\psi \in L^{\infty}(\partial_\mu G,\nu)$, and the map 
\[
\psi \mapsto \widetilde{\psi}(g) = \int_{\partial_\mu G} \psi(g\xi) \, d\nu(\xi)
\]
is an isometry between $L^{\infty}(\partial_\mu G,\nu)$ and the space of bounded harmonic functions on $G$. Since every quasimorphism on an amenable group is trivial in the sense that it is at bounded distance from a homomorphism (see e.g. \cite{scl}) and we are mainly interested in non-trivial quasimorpisms here, there is no loss of generality in assuming $G$ to be non-amenable. In this case $(\partial_\mu G , \nu)$ is indeed a non-atomic Lebesgue space (see e.g. \cite{Fur63}), and we will realize our quasimorphism on this space:
\begin{proposition}\label{IntRepMain}
Let $(G,\mu)$ be a countable measured group, and $\phi$ a quasimorphism on $G$, integrable with respect to $\mu$. Denote by $\tphi$ the unique $\mu$-quasi-biharmonic quasimorphism with $\tphi(e) = 0$ equivalent to $\phi$ and by
$(\partial_\mu G,\nu)$ be the Poisson boundary of $(G,\mu)$. Then there exists a cocycle 
$\alpha : G \times \partial_\mu G \to \bR$ with the following properties:
\begin{itemize}
\item[(i)] $\alpha$ is a cocycle representative for $\phi$ (and hence also for $\tphi$).
\item[(ii)] The quasimorphism $\tphi$ admits the integral representation
\[\tphi(g) = \int_{\partial_\mu G} \alpha(g,\xi) \, d\nu(\xi) \quad(g \in G).\]
\item[(iii)] For $\nu$-almost every $\xi \in \partial_\mu G$ the map $g \mapsto \alpha(g,\xi)$ is quasi-left-harmonic with respect to $\mu$.
\end{itemize}
\end{proposition}

\begin{proof}
The function $\widetilde{\alpha}(g,h) = \tphi(gh) - \tphi(h)$ is right-harmonic with respect to $\mu$ in the second variable, for fixed $g$, 
and uniformly bounded by $\tphi(g) + D(\tphi)$. Thus, for every $g \in G$, there is a conull set $Y_g \subseteq \partial_\mu G$, and a 
measurable function $\alpha(g,\cdot) : Y_g \to \bR$ such that
\[
\tphi(gh) - \tphi(h) = \int_G \alpha(g,h\xi) \, d\nu(\xi), \quad \forall \, h \in G. 
\]
In particular, by letting $h = e$, we conclude that $\tphi(g) = \int_{\partial_\mu G} \alpha(g,\xi) \, d\nu(\xi)$. Since $G$ is countable, we can
without loss of generality assume that $\alpha$ is defined on the set $Y' = \cap_g Y_g$. Since $\tphi$ is quasi-biharmonic, it is clear
that $\alpha(\cdot,\xi)$ is quasi-left-harmonic for all $\xi \in Y'$, so it suffices to establish the cocycle property for $\alpha$ on some
conull subset of $\partial_\mu G$. To prove this, we 
argue as follows. For any triple $g,h,k \in G$, we have
\[
\tphi(ghk) - \tphi(hk) = \tphi(ghk) - \tphi(hk) + \tphi(hk) - \tphi(k), 
\]
and hence, 
\[
\int_{\partial_\mu G}  \alpha(gh,k\xi) \, d\nu(\xi) = \int_{\partial_\mu G}  \alpha(g,hk\xi) \, d\nu(\xi) + \int_{\partial_\mu G}  \alpha(h,k\xi) \, d\nu(\xi).
\]
For a fixed pair $g,h \in G$, we define the following bounded and measurable functions on the Poisson boundary  
$\partial_\mu G$, 
\[
\alpha_0(\xi) := \alpha(gh,\xi) \quad \textrm{and} \quad \alpha_1(\xi) := \alpha(g,h\xi) + \alpha(h,\xi).
\]
By the earlier equality, we see that
\[
\int_{\partial_\mu G} \alpha_0(k\xi) \, d\nu(\xi) = \int_{\partial_\mu G}  \alpha_1(k\xi) \, d\nu(\xi) \quad \forall \, k \in \partial_\mu G,
\]
and both sides define left-harmonic bounded functions on $G$. Since the Poisson boundary 
representation of left-harmonic functions is unique, we conclude that $\alpha_0 = \alpha_1$ 
are equal on a $\nu$-conull subset $Y$ of $\partial_\mu G$, which can be taken to be $G$-invariant and
independent of $g,h \in G$ and contained in $Y'$. The equality $\alpha_0 = \alpha_1$ on $Y$ amounts exactly 
to say that $\alpha$ is a cocycle defined on $Y$. Since $Y$ is $G$-invariant we can extend $\alpha$ to a cocycle
on $\partial_\mu G$ by defining it to be zero on $Y^c$.
\end{proof}
Note that Theorem \ref{IntRepIntro} is contained in Proposition \ref{IntRepMain}.\\

Now that we have established an integral representation for the $\mu$-quasi-biharmonic representative of a given equivalence class of a quasimorphism $\phi$, we would also establish such an integral representation for the homogeneous representative $\hphi$ of this equivalence class. We recall that the latter is given by the formula
\[
\hphi(g) := \lim_n \frac{\phi(g^n)}{n}, \quad g \in G.
\]
It turns out that we can construct an integral representation for $\hphi$ from our integral representation for $\tphi$. This construction has the following remarkable property: If $\alpha$ denotes the cocycle used for the representation of $\tphi$, then there exists an integral kernel for $\hphi$ of the form $\sigma \cdot \alpha$, where $\sigma : G \times X \to \bR$ is some measurable function which does \emph{not} depend on $\phi$. More explicitly:
\begin{proposition} \label{correspondence}
For any countable group $G$ and quasi-invariant $G$-space $(X,\nu)$, there is a measurable map 
$\sigma : G \times X \to \bR$ with the property that for any quasimorphism $\phi$ on the group $G$ and any cocycle representative $\alpha$ of $\phi$ on $(X,\nu)$, we have the identity,
\[
\hphi(g) = \int_{X} \alpha(g,\xi) \, \sigma(g,\xi) \, d\nu(\xi), \quad \, \forall \, g \in G,
\]
where $\hphi$ denotes the homogenization of $\phi$.
\begin{proof}
Since $\nu$ is quasi-invariant under the action by the countable group $G$, the (multiplicative) 
Radon-Nikodyn cocycle, defined by
\[
\rho(g,\xi) := \frac{dg_*\nu}{d\nu}(\xi),
\]
is well-defined on a $\nu$-conull subset $X' \subseteq X$ for all $g \in G$. For a fixed $g$, we 
consider the sequence $(\rho(g^n,\cdot)$ of measurable functions on $X$. We note that 
$\| \rho(g^k,\cdot)\|_1 \leq 1$ for all $n$ and $g$, so we can apply Koml\'os' theorem \cite{Kom67},
and find a subsequence $(n_k)$, which may depend on $g$, such that the limit, 
\[
\sigma(g,\xi) := \lim_N \frac{1}{N} \sum_{k=1}^{N-1} \frac{1}{n_k} \sum_{j=1}^{n_k} \rho(g^{j},\xi)
\]
exists in $L^1(X,\nu)$ and on a conull subset of $X'$. We claim that $\sigma$ is the sought after kernel. 
To prove this, let  $\alpha$ be a cocycle representative of a quasimorphism $\phi$ on $(X,\nu)$, 
and note that 
\begin{eqnarray*}
\int_X \alpha(g,\xi) \, \sigma(g,\xi) \, d\nu(\xi) 
&=& 
\lim_N \frac{1}{N} \sum_{k=1}^{N-1}   \frac{1}{n_k} \sum_{j=1}^{n_k} \int_X \alpha(g,\xi) \, \rho(g^{j},\xi) \, d\nu(\xi) \\
&=&
\lim_N \frac{1}{N} \sum_{k=1}^{N-1} \frac{1}{n_k} \sum_{j=1}^{n_k} \int_X \alpha(g,g^{j}\xi) \, d\nu(\xi) \\
&=&
\lim_N \frac{1}{N} \sum_{k=1}^{N-1} \frac{1}{n_k}  \int_X \alpha(g^{n_k},\xi) \, d\nu(\xi) \\
&=&
\lim_N \frac{1}{N} \sum_{k=1}^{N-1} \frac{\phi(g^{n_k})}{n_k}.
\end{eqnarray*}
where the last equality is justified by the fact that there is a constant $C$ such that 
the essential supremum of $| \alpha(g,\xi) - \phi(g) |$ is bounded by $C$ for all $g \in G$.
Since the Cesaro average of a convergent sequence is convergent with the same limit, 
we can identify the limit with $\hphi(g)$. Finally, since $G$ is countable, the procedure 
above can be repeated for every fixed choice of $g$, to the effect that $\sigma$ is 
defined on a common $\nu$-conull subset of the measure space $(X,\nu)$.
\end{proof}
\end{proposition}
It would be interesting to gain a better
understanding of $\sigma$ from a cohomological point of view. \\

Finally, we note that the (mean) quasi-harmonicity of cocycle representatives for a given measure 
$\mu$ on $G$ is immediate if the $G$-space $(X,\nu)$ is $\mu$-stationary, i.e. if $\nu$ satisfies,
\[
\int_G \int_X \psi(g \xi) \, d\mu(g) d\nu(\xi) = \int_X \psi(\xi) \, d\nu(\xi),
\]
for all $\psi \in L^{\infty}(X,\nu)$. Indeed, let $\alpha$ be a cocycle on $(X,\nu)$, and define
$\phi(g) := \int_X \alpha(g,\xi) \, d\nu(\xi)$ for $g \in G$. Then, 
\begin{eqnarray*}
\int_G \phi(hg) \, d\nu(h) 
&=& 
\int_G \int_X \alpha(hg,\xi) \, d\mu(h) d\nu(\xi) \\
&=&
\int_G \int_X \alpha(h,\xi) \, d\mu(h) d\nu(\xi) + \int_G \int_X \alpha(g,h\xi) \, d\mu(h) d\nu(\xi) \\
&=& \phi(g) + \ell,
\end{eqnarray*}
for all $g \in G$, where $\ell$ is a constant, i.e. $\phi$ is quasi-left-harmonic with respect to $\mu$. In particular, 
if $\alpha$ is a cocycle representative for a quasimorphism $\phi$ on $(X,\nu)$, we conclude that
\[
\ell_\mu(\phi) = \int_G \int_X \alpha(g,\xi) \, \, d\nu(\xi) d\mu(g).
\]
Moreover, if $\phi$ is square-integrable with respect to $\mu$, a simple argument gives
\[
\sigma^2 = \int_G \int_X ( \alpha(g,\xi) - \ell_\mu(\phi) )^2 \,  d\nu(\xi) d\mu(g),
\] 
where $\sigma^2$ is the variance in Theorem \ref{main}.

\appendix
\section{Strongly affine maps and Mokobodski means} \label{Mokobodzki theory}

This section is devoted to the construction of Mokobodzki means, a special class of invariant means on $\bN$ (or any amenable semigroup),
which are universally measurable and satisfy an analogue of Lebesgue's dominated convergence theorem for uniformly bounded sequences
of measurable functions. We stress that the construction can only be made within certain extensions of ZFC, e.g. under the assumption on the 
Continuum Hypothesis \cite{Mey73} or under the (weaker) assumption of Martin's axiom \cite{Nor76}. It has recently been proved \cite{Lar10} 
that Mokobodzki means do \emph{not} exist if the Continuum Hypothesis is discarded and replaced by the Filter Dichotomy. The exposition given
here is expanded from the classical source \cite{Mey73}.

Let $(Y,\cB)$ be a measurable space. A set $A \subseteq Y$ is \emph{universally measurable} if it is measurable
with respect to every complete finite measure on the $\cB$. If $(Z,\cC)$ is another 
measurable space, we say that a map $f : Y \to Z$ is universally measurable if the pre-image of every 
universally measurable set in $Z$ is universally measurable in $Y$. We denote by $\cB^*$ the $\sigma$-algebra
of all universally measurable sets in $(Y,\cB)$. 

Let $E$ be a locally convex and separable topological vector space, and $K$ a compact and convex metrizable subset of $E$. We denote
by $\cC^+(K)$ the set of convex functions on $K$, which are bounded from below and equal to the limit of a decreasing \emph{sequence}
of convex lower semicontinuous functions, and we define $\cC^{-}(K) := - \cC^+(K)$. We say that a real-valued function $u$ on $K$ is 
\emph{strongly convex} if it is universally measurable, bounded from below, and if for all $x \in K$ and all probability measures $\nu$ 
on $K$ with barycenter $x$, we have
\[
u(x) \leq \int_K u(y) \, d\nu(y).
\]
Similarly, a real-valued function $v$ on $K$ is \emph{strongly concave} if $-v$ is strongly convex, and \emph{strongly affine} if it is both 
strongly convex and concave. 

The goal of this appendix is to prove the following beautiful result by Mokobodzki \cite{Mey73}.
\begin{proposition}[CH] \label{mokobodski}
There is an invariant mean $\gom$ on $\bN$ with the property that whenever $(\Omega,\cB)$ is a measurable space, and 
$(\psi_n)$ a uniformly bounded sequence of $\cB$-measurable functions on $\Omega$, the function
\[
\psi^* := \int_{\bN} \psi_n \, d\gom(n)
\]
is $\cB^{*}$-measurable and moreover, the equality
\[
\int_\Omega \psi^{*} \, d\lambda = \int_{\bN} \Big( \int_\Omega \psi_n \, d\lambda \Big) \, d\gom(n)
\]
holds for every probability measure $\lambda$ on $(\Omega,\cB)$.
\end{proposition}

We need the following two lemmata. 

\begin{lemma} \label{mokobodski help lemma}
Suppose that $u \in \cC^{-}(K)$ and $v \in \cC^+(K)$ with $u \leq v$ on $K$ are given. Then, for any probability measure 
$\nu$ on $K$, there are $u' \in \cC^{-}(K)$ and $v' \in \cC^{+}(K)$ such that the inequalities $u \leq u' \leq v' \leq v$ hold 
and the equality $u' = v'$ holds almost everywhere on $K$ with respect to $\nu$.
\end{lemma}

\begin{proof}
By the definition of the class $\cC^{+}(K)$, there is is a decreasing sequence $(v_n)$ of convex lower semicontinuous functions,
which we can assume is bounded from below, such that $v = \inf_n v_n$. Similarly, there is an increasing sequence $(u_n)$ of
concave upper semicontinuous functions, bounded from above, such that $u = \sup_n u_n$.  Define 
\[
u_n' = u_n - \frac{1}{n} \quad \textrm{and} \quad v_n' = v_n + \frac{1}{n}, 
\]
and note that by the Hahn-Banach theorem, the set $L_n$ of continuous and affine maps on $K$ which are sandwiched 
between $u_n'$ and $v_n'$ is non-empty for all $n$. Since both $u_n'$ and $v_n'$ 
are bounded functions on $K$, the weak closure $\overline{L}_n$ of $L_n$ in $L^1(K,\nu)$ is compact for all $n$. Moreover, the 
family $(L_n)$ is a nested sequence of pre-compact and convex subsets of $L^1(K,\nu)$, so we conclude that the countable intersection 
$\cap_n \overline{L}_n$ contains at least one element, which we denote by $w'$. By the Banach-Saks theorem, we can choose a sequence 
$(w_k)$ with $w_k \in \overline{L}_{n_k}$, for some sequence $(n_k)$, such that $w_k \to w'$ in the norm topology of $L^1(K,\mu)$ 
and moreover, $\| w_k - w' \|_1 \leq 2^{-k}$. By a standard application of Borel-Cantelli's lemma, the sequence $(w_k)$ converges 
to $w'$ on a conull subset $K'$ of $K$, so we may define the functions
\[
u' = \liminf_k w_k \quad \textrm{and} \quad v' = \limsup_k w_k.
\]
By construction, we have the inequalities $u \leq u' \leq v' \leq v$ on $K$, and $u' = v'$ on $K'$. 
\end{proof}

The second lemma uses the Continuum Hypothesis. Before we state and prove the lemma, we recall some basic notions from set theory. 
The set of \emph{countable ordinals} $(\Omega,<)$ is the unique (up to order isomorphism)  well-ordered uncountable set with the property that 
the sets
\[
\Omega_x := \{ y \in \Omega \, |\, y < x \}
\]
are countable, for all $x \in \Omega$. We stress that the construction of $(\Omega,<)$ only depends on the Axiom of Choice. 
The cardinality of $\Omega$ is $\aleph_1$, so by the Continuum Hypothesis there is a set-theoretic bijection between $\Omega$
and the real line. 

\begin{lemma}[CH] \label{mokobodski key lemma}
Suppose that $u \in \cC^{-}(K)$ and $v \in \cC^+(K)$ with $u \leq v$ on $K$ are given. Then there is a strongly affine map $w$ 
on $K$ which satisfies $u \leq w \leq v$.
\end{lemma}

\begin{proof}
Since the set of all probability measures on $K$ has the same cardinality as the real line (unless $K$ consists of one point, in which case
the lemma is trivial), we can index this set by the countable ordinals, i.e. $M^1(K) = (\nu_\alpha)_{\alpha \in \Omega}$. By a transfinite 
recursion of Lemma \ref{mokobodski help lemma}, we produce a decreasing sequence $(v_\alpha)$ of elements in $\cC^+(K)$, and an
increasing family $(u_\alpha)$ of elements in $\cC^{-}(K)$, such that
\[
u \leq u_\alpha \leq v_\alpha \leq v,
\]
and $u_\alpha = v_\alpha$ $\nu_\alpha$-almost everywhere. Since all probability measures on $K$ occurs in the enumeration, we have
$\inf_\alpha v_\alpha = \sup_\alpha u_\alpha$. We denote this function by $w$. As we have $w = u_\alpha = v_\alpha$ on a 
$\nu_\alpha$-conull subset of $K$, the function $w$ is universally measurable, and it is clearly bounded. 

The barycenter property is easily verified using the following simple trick. Let $\nu$ be a probability measure on $K$ with barycenter $x$. 
Then there is a countable ordinal $\alpha$ with $\nu_\alpha = ( \nu + \delta_x) / 2$. As $w$ is $\nu_\alpha$-almost everywhere equal to
$u_\alpha$, which is an upper semicontinuous concave function, and also to $v_\alpha$, which is a lower semicontinuous convex function, 
we have $w(x) = \int_K w(y) \, d\nu(y)$, and hence $w$ is strongly affine.
\end{proof}

We can now prove Proposition \ref{mokobodski}.

\begin{proof}[Proof of Proposition \ref{mokobodski}]
Let $K$ denote the unit ball in $\ell^{\infty}(\bN)$. We can without loss of generality assume that 
$|\psi_n| \leq 1$ for all $n$, so that the map $\tilde{\psi} := (\psi_n)$ is well-defined and $\cB$-measurable 
as a map of $\Omega$ into $K$. The functions
\[
u(x) := \liminf_n \, \frac{1}{n} \sum_{k=1}^{n} x_k \quad \textrm{and} \quad v(x) := \limsup_n \, \frac{1}{n} \sum_{k=1}^{n}  x_k
\]
clearly belongs to the classes $\cC^-(K)$ and $\cC^+(K)$ respectively, so by Lemma \ref{mokobodski key lemma},
there is a strongly affine and shift-invariant map $\gom$ on $K$ which is sandwiched by $u$ and $v$. Hence, the 
composition $\psi^* := \int_{\bN} \psi_n \, d\gom(n)$ is universally measurable, which proves the first part of the proposition. 

Moreover, if $\lambda$ is a probability measure on $\Omega$, then the push-forward $\widetilde{\psi}_*\lambda$ is 
a probability measure on $K$, so by the defining property of strongly affine maps, we have
\[
\int_K \Big( \int_{\bN} x_n \, d\gom(n) \Big) \, d\widetilde{\psi}_*\lambda(x) = \int_{\bN} \Big( \int_K x_n \, d\widetilde{\psi}_*\lambda(x) \Big) \, d\gom(n),
\]
which exactly translates to the equality asserted in the proposition. 
\end{proof}

\end{document}